\documentclass[a4paper,11pt]{article}

\usepackage[utf8]{inputenc}
\usepackage[english]{babel}
\usepackage{amsmath,amssymb,amsthm,,amsbsy,bbm} % amsfonts
\usepackage{hyperref,xcolor,fullpage}
\usepackage[numbers]{natbib} 

\usepackage{titlesec}
\makeatletter \let\@fnsymbol\@arabic \makeatother % don't like that cross next to names
\titleformat{\section}[block]
	{\normalfont\large\bfseries\centering}
	{\thesection}{.5em}{}
\titleformat{\subsection}[runin]
	{\normalfont\bfseries}
	{\thesubsection}{.5em}{}
\titleformat{\subsubsection}[runin]
	{\normalfont\bfseries}
	{\thesubsubsection}{.5em}{}
 
\allowdisplaybreaks % allows break in multiline formulas
\newtheorem{theorem}{Theorem}[section]
\newtheorem{lemma}[theorem]{Lemma}
\newtheorem{definition}[theorem]{Definition}

\newtheorem{conjecture}[theorem]{Conjecture}
\newtheorem{claim}[theorem]{Claim}
\newtheorem{proposition}[theorem]{Proposition}

\newtheorem{example}[theorem]{Example}
\newtheorem{problem}[theorem]{Problem}

\newcommand{\bx}{\textbf{x}}
\newcommand{\mS}{\mathcal{S}}
\newcommand{\mH}{\mathcal{H}}
\newcommand{\mF}{\mathcal{F}}
\newcommand{\mQ}{\mathcal{Q}}

\title{Shattering-extremal set systems from  Sperner families}
\date{}

\author{
	Christopher Kusch \thanks{Freie Universit\"at Berlin, Institut f\"ur Mathematik und Informatik, Arnimallee 3, 14195 Berlin, Germany and Berlin Mathematical School, Germany. E-mail: {\tt c.kusch@gmx.net}. Supported by Berlin Mathematical Phase II scholarship.} \and
	T\'amas M\'esz\'aros \thanks{Freie Universit\"at Berlin, Institut f\"ur Mathematik und Informatik, Arnimallee 3, 14195 Berlin, Germany and Berlin Mathematical School, Germany. E-mail: {\tt tamas.meszaros@fu-berlin.de}. Research is supported by the DRS POINT Postdoc Fellow program }
}
\begin{document}

\maketitle

\begin{abstract}
We say that a set system $\mathcal{F}\subseteq 2^{[n]}$ \emph{shatters} a given set $S\subseteq [n]$ if $2^S= \{F~\cap~S:~F~\in~\mathcal{F}\}$.
  The Sauer-Shelah lemma states that in general, a set system $\mathcal{F}$ shatters at least $|\mathcal{F}|$ sets. Here we concentrate on the case of equality. A set system is called \emph{shattering-extremal} if it shatters exactly $|\mathcal{F}|$ sets.  A conjecture of R\'onyai and the second author and of Litman and Moran states that if a family is shattering-extremal then one can add a set to it and the resulting family is still shattering-extremal. Here we prove this conjecture for a class of set systems defined from Sperner families. 
 \end{abstract}

%%%%%%%%%%%%%%%%%%%%
%%%%%%%%%%%%%%%%%%%%
%%%%%%%%%%%%%%%%%%%%  INTRODUCTION
%%%%%%%%%%%%%%%%%%%%
%%%%%%%%%%%%%%%%%%%%

\section{Introduction}

Let $n\in \mathbb{N}$ and set $[n]=\{1,...,n \}$. If $X\subseteq [n]$ and $I\subseteq [n]\setminus X$, we write $2^X$ to denote the power set of $X$, $I+2^X$ for the family $\{I\cup A\ :\ A\subseteq X\}$ and ${X\choose k}$ for the collection of subsets of $X$ of size $k$. A set system $\mathcal{F} \subseteq 2^{[n]}$ is a \emph{down-set} (\emph{up-set}) if $G\subseteq F$ and $F\in \mathcal{F}$ ($G\in \mathcal{F}$) implies $G\in \mathcal{F}$ ($F\in \mathcal{F}$). 
\begin{definition}
	A set system  \emph{shatters} a given set $S \subseteq[n]$ if $2^S = \{F \cap S\; : \; F \in \mathcal{F}  \}$. The family of subsets of $[n]$ shattered by $\mathcal{F}$ is denoted by $\text{Sh}(\mathcal{F})$. 
\end{definition}
We remark that the notion of shattering can also be stated in terms of the \emph{trace} of a set system. Given a set $S\subseteq [n]$, the \emph{trace} $\mF|_{S}$ of a set system $\mF$ on $S$ is defined as $\mF|_S = \{F \cap S \; : \; F \in \mF \}$. Then $S$ is shattered by $\mF$ precisely if $\mF|_S = 2^S$. This naturally leads to certain \emph{forbidden trace problems}, for which we refer the interested reader to the survey of F\"uredi and Pach~\cite{FP92}.

A natural first question is to ask how the size of a family $\mF$ relates to the size of the family it shatters. This question is answered from one side in the following fundamental result, which is usually referred to as the \emph{Sauer-Shelah lemma}.

\begin{proposition}\label{S-Sh-l}
$|\text{Sh}(\mathcal{F})|\geq |\mathcal{F}|$ for every set system
$\mathcal{F}\subseteq 2^{[n]}$.
\end{proposition}

This statement was proved by several authors independently, for a proof see e.g. \cite{ARS02}. Here we are interested in the case of equality. 

\begin{definition}
A set system $\mathcal{F}\subseteq 2^{[n]}$ is \emph{shattering-extremal}, or \emph{s-extremal} for short, if it
shatters exactly $|\mathcal{F}|$ sets, i.e. $|\mathcal{F}|=|\text{Sh}(\mathcal{F})|$.
\end{definition} For example, if $\mathcal{F}$ is a down-set then $\mathcal{F}$ is s-extremal, simply because in this case  $\text{Sh}(\mathcal{F})= \mathcal{F}$. Many interesting results have been obtained in connection with these combinatorial objects, among others by Bollob{\'a}s, Leader and Radcliffe in \cite{BLR89}, by Bollob{\'a}s and Radcliffe in \cite{BR95}, by Frankl in \cite{Frankl96} and recently Kozma and Moran in \cite{KM13} provided further interesting examples of s-extremal set systems. Anstee, R{\'o}nyai and Sali in \cite{ARS02} related shattering to standard monomials of vanishing ideals, and based on this, M\'esz\'aros and R\'onyai in \cite{RM11} developed algebraic methods for the investigation of s-extremal families, which we will briefly recall later.

To broaden the picture, we now mention some well known related results. The \emph{Vapnik-Chervonenkis dimension} of $\mathcal{F}$, denoted by $\text{dim}_{VC}(\mathcal{F})$, is the size of the largest set shattered by $\mathcal{F}$. An easy corollary of the Sauer-Shelah lemma is the following result, known as the Sauer-inequality, which has found applications in a variety of contexts.

\begin{proposition}[\cite{Sauer72},\cite{Shelah72}]\label{S-i}
Let $0\leq k\leq n$ and $\mathcal{F}\subseteq 2^{[n]}$. If $\mathcal{F}$ shatters no set of size $k$, i.e. $\text{dim}_{VC}(\mathcal{F})\leq k-1$, then
\begin{equation}\label{S-i-eq}
|\mathcal{F}|\leq \sum_{i=0}^{k-1} {n\choose i}.
\end{equation}
\end{proposition}

Families satisfying (\ref{S-i-eq}) with equality are called \emph{maximum classes}, and serve as important examples in the theory of machine learning. They have several nice properties, among others they are s-extremal. In the case of uniform families the above bound can be strengthened.

\begin{proposition}[\cite{FP94}]\label{F-P}
Let $0\leq k\leq l\leq n$ and $\mathcal{F}\subseteq {[n]\choose l}$. If $\mathcal{F}$ shatters no set of size $k$, i.e. $\text{dim}_{VC}(\mathcal{F})\leq k-1$, then
\begin{equation*}
|\mathcal{F}|\leq {n\choose k-1}.
\end{equation*}
\end{proposition}

A set family $\mathcal{S} \subseteq 2^{[n]}$ is called a \emph{Sperner family}, or an \emph{antichain}, if none of its sets is contained in another. In connection with Proposition~\ref{F-P} it is an interesting open problem due to Frankl~\cite{Frankl89} whether the above bound holds for Sperner families in general not merely uniform ones. Sperner families will play an important role in our study of s-extremal set systems as well, since one can use them to `build' s-extremal set systems.

\medskip

The main goal in connection with s-extremal families is to find good characterisations of them. A positive answer to the following conjecture, formulated in \cite{MR14}, would be a possible way for this.

\begin{conjecture}\label{conj}
For every s-extremal set system $\mathcal{F} \subsetneq 2^{[n]}$ there exists $F \notin \mathcal{F}$ such that $\mathcal{F} \cup \{F\}$ is again s-extremal.
\end{conjecture}  

As by Theorem 2 in \cite{BR95} $\mathcal{F}$ is s-extremal if and only if $2^{[n]}\setminus  \mathcal{F}$ is so, the above conjecture has an equivalent form, namely that for every non-empty s-extremal set system $\mathcal{F} \subseteq 2^{[n]}$ there exists $F \in \mathcal{F}$ such that $\mathcal{F} \setminus \{F\}$ is again s-extremal. It will be always clear from the context which form of the conjecture we consider. This latter form was formulated by Litman and Moran independently, and called the corner peeling conjecture. For maximum classes essentially the same was conjectured by Kuzmin an Warmuth in \cite{KW07} and proven by Rubinstein and Rubinstein in \cite{RR12}. There are several other cases when the conjecture is known to be true. First of all it is trivially true for down-sets, as there you can always add any minimal element not belonging to it. M\'esz\'aros and R\'onyai in \cite{MR13} and \cite{MR14}, using a graph theoretic approach, proved the conjecture for s-extremal families of VC-dimension at most $2$. According to personal communication, the same result was independently proven by Litman and Moran. Some examples of Anstee in \cite{A80} and of F\"uredi and Quinn in \cite{FQ83} also turned out to be s-extremal and they also satisfy the conjecture. According to Moran and Warmuth \cite{MW16} the conjecture, if true, would imply unlabeled compression schemes for s-extremal classes, which so far were known to exist for maximum classes.

\subsection*{Results.} In order to state our main results, we first introduce some notation.  Let $\mathcal{S}\subseteq 2^{[n]}$ be a Sperner family and let  $h:\mathcal{S}\rightarrow 2^{[n]}$ be a function such that $h(S)\subseteq S$ for every $S\in \mathcal{S}$. For $H\subseteq S\subseteq [n]$ define
\begin{equation*}
\mathcal{P}_{S} = S + 2^{[n]\setminus S}\text{ and  }\mathcal{Q}_{S,H} = H + 2^{[n]\setminus S}.
\end{equation*}
Note that $\mathcal{P}_{S}$ and $\mathcal{Q}_{S,H}$ are hypercubes of the same dimension, $n-|S|$, in particular $|\mathcal{P}_{S}|=|\mathcal{Q}_{S,H}|$, $\mathcal{P}_S$ is the collection of all sets containing $S$ and $\mQ_{S,H}$ is the collection of all sets whose intersection with $S$ is $H$. We define the up-set generated by $\mathcal{S}$ as
\begin{equation*}
\text{Up}(\mathcal{S})=\{F\subseteq [n]\ :\ \exists S\in \mathcal{S}\text{ such that }S\subseteq F\}.
\end{equation*}
Note that $\text{Up}(\mathcal{S}) = \bigcup_{S\in \mathcal{S}}\mathcal{P}_S $. Further set
\begin{equation*}
\mathcal{H}(\mathcal{S}) = 2^{[n]}\setminus \text{Up}(\mathcal{S})=2^{[n]}\setminus \bigcup_{S\in \mathcal{S}}\mathcal{P}_S\ \text{ and}
\end{equation*}
\begin{equation*}
\mathcal{F}(\mathcal{S},h) = 2^{[n]}\setminus \bigcup_{S\in \mathcal{S}} \mathcal{Q}_{S,h(S)}.
\end{equation*}

At this point we would like to remark that if $\mS$ is a general family, not necessarily a Sperner family, then, by passing to the collection of minimal elements in $\mS$, most of our results remain true/can still be formulated. However our interest mostly lies in the case of Sperner families and so for simplicity we will consider only them.

 The following proposition is the starting point for our discussion which might be a good first step towards a nice characterisation of s-extremal families. 
\begin{proposition}\label{prop::s-extremal}
Let $\mathcal{S}\subseteq 2^{[n]}$ be a Sperner family and let  $h:\mathcal{S}\rightarrow 2^{[n]}$ be a function such that $h(S)\subseteq S$ for every $S\in \mathcal{S}$. Then $\mathcal{F}=\mathcal{F}(\mathcal{S},h)$ is s-extremal with $\text{Sh}(\mF) = \mH(\mS)$ if and only if
\begin{equation}\label{s-extremal_condition}
\left| \mF(\mS,h)\right|=\left|  \mH(\mS) \right|.
\end{equation}
\end{proposition}
The reason this really might be a good starting point to tackle the elimination conjecture is the following lemma.

\begin{lemma}\label{unique lemma}
Let $\mF \subseteq 2^{[n]}$ be an s-extremal family. Then there is a unique Sperner family $\mS$ and a unique function $h:\mathcal{S}\rightarrow 2^{[n]}$ with $h(S)\subseteq S$ for every $S\in \mathcal{S}$ such that $\mF=\mF(\mS,h)$ and $\text{Sh}(\mF) = \mH(\mS)$.
\end{lemma}

\smallskip

We will study the applications of Proposition~\ref{prop::s-extremal} in three different ways. Firstly we will prove Conjecture~\ref{conj} for a special class motivated by Equation~\ref{s-extremal_condition}. More precisely, we will show the following theorem.
\begin{theorem}\label{thm::elimination_main}
	Let $\mathcal{S}\subseteq 2^{[n]}$ be a Sperner family and $A \subseteq [n]$ a fixed set. Furthermore let $h_A:\mathcal{S}\rightarrow 2^{[n]}$ be defined as $h_A(S)=S\cap A$. Then $\mF(\mS, h_A)$ is s-extremal and Conjecture \ref{conj} holds for $\mathcal{F}(\mathcal{S},h_A)$, i.e. there exists $F\notin \mathcal{F}(\mathcal{S},h_A)$ such that $\mathcal{F}'=\mathcal{F}(\mathcal{S},h_A)\cup\{F\}$ is again s-extremal. Moreover $\mathcal{F}'=\mathcal{F}(\mathcal{S}',h_A)$ for some suitable Sperner family $\mathcal{S}'$.
\end{theorem}
Secondly, we will prove Conjecture~\ref{conj} when the Sperner family corresponding to the s-extremal family is small.  
\begin{theorem}\label{thm::conjecture_for_small_sperner}
	Let $\mS \subseteq 2^{[n]}$ be a Sperner family of size at most four, $h:\mathcal{S}\rightarrow 2^{[n]}$ a function such that $h(S)\subseteq S$ for every $S\in \mathcal{S}$ and suppose that the resulting family $\mF(\mS,h)$ is s-extremal. Then Conjecture~\ref{conj} holds for $\mF(\mS,h)$, i.e. there is $F\notin \mathcal{F}(\mathcal{S},h)$ such that $\mathcal{F}'=\mathcal{F}(\mathcal{S},h)\cup\{F\}$ is again s-extremal.
\end{theorem}
For this, we shall present an equivalent form of the conjecture which is formulated in terms of the cubes $\mQ_{S,h(S)}$.

Lastly we continue the study of the connection between s-extremal families and so-called Gr\"obner bases, initiated by M\'esz\'aros and R\'onyai~\cite{RM11}. Since the result requires some more definitions we will only state it after introducing Gr\"obner bases in Section~\ref{sec::preliminaries_groebner}.

%%%%%%%%%%%%%%%%%%%%
%%%%%%%%%%%%%%%%%%%%
%%%%%%%%%%%%%%%%%%%%  Approach
%%%%%%%%%%%%%%%%%%%%
%%%%%%%%%%%%%%%%%%%%

\section{An approach to the elimination conjecture}

\subsection{Proof of Proposition~\ref{prop::s-extremal} and Lemma \ref{unique lemma}}

\begin{proof}[Proof of Proposition~\ref{prop::s-extremal}]
Suppose $\mF=\mF(\mS,h)$ is s-extremal with Sh$(\mF) = \mH(\mS)$. Then, by the s-extremality we have $|\mF| = | \text{Sh}(\mF)| = \left| \mH(\mS) \right|$ as claimed. 

\medskip

To see the other direction, suppose that $ |\mathcal{H}(\mathcal{S})|=|\mathcal{F}(\mathcal{S},h)| $. 

\begin{claim}\label{fshsh}
Let $\mS \subseteq 2^{[n]}$ be a Sperner family and $h:\mathcal{S}\rightarrow 2^{[n]}$ a function such that $h(S)\subseteq S$ for every $S\in \mathcal{S}$. Then $\text{Sh}(\mF(\mS,h))\subseteq \mH(\mS)$.
\end{claim}
\begin{proof}
Note that by the definition of $\mF=\mF(\mS,h)$, for every $S \in \mS$ there does not exist $F \in \mF$ such that $F\cap S = h(S)$ and so $S \notin \text{Sh}(\mF)$. In particular, no superset of $S$ is shattered by $\mF$. Therefore, $\text{Sh}(\mF) \subseteq 2^{[n]}\setminus \text{Up}(\mS) = \mH(\mS)$. 
\end{proof}

In our case this in particular means  that $|\text{Sh}(\mF(\mS,h))| \leq |\mH(\mS)| =|\mF(\mS,h)|$. However, the reverse inequality always holds by the Sauer-Shelah lemma and so $\mF(\mS,h)$ is s-extremal and $\text{Sh}(\mF(\mS,h)) = \mH(\mS)$ as claimed.
\end{proof}
Although this result is rather easy to state and prove, it does offer a new perspective to s-extremal set systems, because it allows one to construct an s-extremal set system from a Sperner family with an appropriately defined function $h$. Given a Sperner family and a function $h$, one checks whether Equation~\eqref{s-extremal_condition} is satisfied and if it is, the resulting set system is s-extremal. 

\begin{proof}[Proof of Lemma~\ref{unique lemma}]
First note that for every up-set there is a unique Sperner family, namely the collection of all minimal elements, that generates it. In particular, using this fact for $2^{[n]}\setminus \text{Sh}(\mF)$ we get that $\mS$, if it exists, is really unique and it is the collection of all minimal sets not shattered by $\mF$. Fix this $\mS$ a keep in mind that we have $\text{Sh}(\mF)=\mH(\mS)$.

\medskip

In \cite{M10} it was proven that if $\mF$ is s-extremal and $S$ is a minimal set not shattered by $\mF$ then the subset of $S$ that cannot be obtained as the intersection with some member of $\mF$ is unique. Accordingly fix $h$ to be the function which maps every $S\in \mS$ to this unique subset. Now this means that for every $S\in \mS$ we have $\mF\subseteq 2^{[n]}\setminus \mQ_{S,h(S)}$, in particular $\mF\subseteq \mF(\mS,h)$. However then by the Sauer-Shelah lemma, Claim~\ref{fshsh}, the choice of $\mS$ and the s-extremality of $\mF$ we have $|\mF|\leq |\mF(\mS,h)|\leq |\text{Sh}(\mF(\mS,h))|\leq |\mH(\mS)|=|\text{Sh}(\mF)|=|\mF|$, implying that actually $\mF= \mF(\mS,h)$.

\medskip

Now suppose that there is another function $h'$ for which $\mF=\mF(\mS,h')$ and take an arbitrary $S\in \mS$. By construction $\mF=\mF(\mS,h')=2^{[n]}\setminus \bigcup_{S'\in \mathcal{S}} \mathcal{Q}_{S',h'(S)}\subseteq 2^{[n]}\setminus \mathcal{Q}_{S,h'(S)}$, meaning that there is no set $F\in \mF$ with $F\cap S_0=h'(S_0)$. However by our earlier remark we know that this subset of $S_0$ is unique so we must have $h'(S)=h(S)$ for every $S\in \mS$.
\end{proof}

In order to try to justify our approach given by Proposition~\ref{prop::s-extremal} further, we remark that it has a connection to the following generalisation of the Sauer inequality, to which our attention was brought by Chornomaz \cite{C16} and which we implicitly already proved in the proof of Proposition~\ref{prop::s-extremal}. For the sake of completeness below we shortly repeat the argument.

\begin{proposition}\label{genSauer}
Let $\mathcal{S}\subseteq 2^{[n]}$ be a Sperner family and $\mathcal{F}\subseteq 2^{[n]}$ a set system that shatters no element of $\mathcal{S}$. Then
\begin{equation*}
|\mathcal{F}|\leq |\mathcal{H}(\mathcal{S})|.
\end{equation*}
\end{proposition}

\begin{proof} For the proof just note that if $\mathcal{F}$ shatters no element of $\mathcal{S}$, then it shatters no set from $\text{Up}(\mathcal{S})$ either, and so $\text{Sh}(\mathcal{F})\subseteq 2^{[n]}\setminus \text{Up}(\mathcal{S})$. Accordingly by the Sauer-Shelah lemma
\begin{equation*}
|\mathcal{F}|\leq |\text{Sh}(\mathcal{F})|\leq |2^{[n]}\setminus \text{Up}(\mathcal{S})|=|\mathcal{H}(\mathcal{S})|
\end{equation*}
as wanted.
\end{proof}

For a Sperner family $\mathcal{S}\subseteq 2^{[n]}$ let us define a family $\mathcal{F}\subseteq 2^{[n]}$ shattering no element of $\mathcal{S}$ and satisfying $|\mathcal{F}|=|\mathcal{H}(\mathcal{S})|$ to be \emph{$\mathcal{S}$-extremal}. Note that the original Sauer inequality can be recovered by setting $\mathcal{S}={[n]\choose k}$, and ${[n]\choose k}$-extremal families are just maximum classes. An interesting property here is that if we let $\mathcal{S}$ to vary then we end up with s-extremality.

\begin{proposition}
$\mathcal{F}\subseteq 2^{[n]}$ is s-extremal if and only if there exists a Sperner family $\mathcal{S}\subseteq 2^{ [n]}$ such that $\mathcal{F}$ is $\mathcal{S}$-extremal.
\end{proposition}
\begin{proof}
First suppose that $\mathcal{F}\subseteq 2^{[n]}$ is s-extremal. By Lemma~\ref{unique lemma} there is a unique Sperner family $\mS$, namely the minimal sets not shattered by $\mF$, and a function $h:\mathcal{S}\rightarrow 2^{[n]}$ with $h(S)\subseteq S$ for every $S\in \mathcal{S}$ such that $\mF=\mF(\mS,h)$ and $\text{Sh}(\mathcal{F})=\mathcal{H}(\mathcal{S})$. By the s-extremality this implies $|\mathcal{F}|=|\text{Sh}(\mF)|=|\mathcal{H}(\mathcal{S})|$.  On the other hand by the choice of $\mS$  its elements are not shattered by $\mathcal{F}$ hence it is $\mathcal{S}$-extremal.

\medskip

Now suppose that $\mathcal{F}$ is $\mathcal{S}$-extremal for some Sperner family $\mathcal{S}\subseteq 2^{[n]}$. From the proof of Proposition~\ref{genSauer} if follows that this is possible only if $\text{Sh}(\mathcal{F})=\mathcal{H}(\mathcal{S})$. However this means that $|\text{Sh}(\mathcal{F})|=|\mathcal{H}(\mathcal{S})|=|\mathcal{F}|$ and so $\mathcal{F}$ is s-extremal.
\end{proof}

\subsection{Analysing the equation $|\mH(\mS)| = |\mF(\mS,h)|$}\label{equiv equation}

Let us now get back to families of the form $\mathcal{F}(\mathcal{S},h)$.  For a Sperner family $\mathcal{S}=\{S_1,\dots,S_N\}$ and a function $h:\mS\rightarrow 2^{[n]}$ with $h(S)\subseteq S$ for every $S\in \mS$, to simplify notation, put $h(S_i)=H_i$. To analyse \eqref{s-extremal_condition} in Proposition~\ref{prop::s-extremal} first note that it holds if and only if $2^{[n]}\setminus \mH(\mS)=\text{Up}(\mathcal{S})=\bigcup_{i=1}^N\mathcal{P}_{S_i}$ and $2^{[n]}\setminus \mathcal{F}(\mathcal{S},h)=\bigcup_{i=1}^N\mathcal{Q}_{S_i,H_i}$ have the same size. To study this we will use the inclusion-exclusion formula. For this note that for every $1\leq i < j \leq N$
\begin{itemize}
	\item[(i)]  $\mathcal{P}_{S_i} \cap \mathcal{P}_{S_j} = \mathcal{P}_{S_i \cup S_j}$ and 
	 \item[(ii)] $ \mathcal{Q}_{S_i,H_i}\cap \mathcal{Q}_{S_j,H_j} = \begin{cases}
\mathcal{Q}_{S_i\cup S_j, H_i \cup H_j}& \text{ if } S_i \cap H_j = S_j\cap H_i\\
\emptyset &\text{ otherwise} 
\end{cases}$.
\end{itemize}
For $I\subseteq [N]$ put $S_I=\bigcup_{i\in I}S_i$ and $H_I=\bigcup_{i\in I}H_i$. Then in particular we have that
\begin{equation*}
\left|\bigcap_{i\in I}\mathcal{P}_{S_i}\right|=\left|\mathcal{P}_{S_I}\right|=\left |\mathcal{Q}_{S_I,H_I}\right|=\left|\bigcap_{i\in I}\mathcal{Q}_{S_i,H_i}\right|
\end{equation*}
whenever $\bigcap_{i\in I}\mathcal{Q}_{S_i,H_i}$ is non-empty, which happens exactly if for every $i\neq j\in I$ we have $ S_i \cap H_j = S_j\cap H_i$. For $1\leq i,j\leq N$ let $\mathbb{I}_{i,j}$ be the indicator of the event $ S_i \cap H_j = S_j\cap H_i$, i.e. it is $1$ if the equality is satisfied and $0$ otherwise. Now the inclusion-exclusion formula gives that we have $|\mathcal{H}(\mathcal{S})|=|\mathcal{F}(\mathcal{S},h)|$ if and only if
\begin{equation*}
\sum_{I\subseteq [N]}(-1)^{|I|+1}\left|\bigcap_{i\in I}\mathcal{P}_{S_i}\right|=\sum_{I\subseteq [N]}(-1)^{|I|+1}\left|\bigcap_{i\in I}\mathcal{Q}_{S_i,H_i}\right|=\sum_{I\subseteq [N]}(-1)^{|I|+1}\left(\prod_{i\neq j\in I}\mathbb{I}_{i,j}\right)\left|\bigcap_{i\in I}\mathcal{P}_{S_i}\right|.
\end{equation*}
This latter equation can also be rewritten as 
\begin{equation*}
\sum_{I\subseteq [N]}(-1)^{|I|+1}\left(1-\prod_{i\neq j\in I}\mathbb{I}_{i,j}\right)\left|\bigcap_{i\in I}\mathcal{P}_{S_i}\right|=0.
\end{equation*}

\subsection{Outline of the new approach}

Before we prove Theorem~\ref{thm::elimination_main} we would like to begin with a high overview of our approach to the elimination conjecture. 
\begin{itemize}
	\item[(i)] Start with an s-extremal family $\mF$. Then we know that there is a unique Sperner family $\mS \subseteq 2^{[n]}$ and a unique function $h: \mS \rightarrow 2^{[n]}$ with $h(S)\subseteq S$ for every $S\in \mS$ such that $\mF=\mF(\mS,h)$ and $\text{Sh}(\mF)=\mH(\mS)$.
	\item[(ii)] Choose a suitable set $S_0 \in \mS$ and replace it with sets from $\{ S_0\cup \{v\}\; : v\in [n]\setminus S_0  \}$ to obtain a Sperner family $\mS'$ with $\mH(\mS') = \mH(\mS) \cup \{S_0\}$. Note that it might be the case that $\mS' = \mS \setminus\{S_0 \}$.
	\item[(iii)] Extend the function $h$ suitably from $\mS$ to $\mS'$ and consider the resulting set system $\mF' = \mF(\mS',h)$. Note that by the Sauer-Shelah lemma, Claim~\ref{fshsh} and the choice of $\mS'$ we automatically have $|\mF'| \leq |\text{Sh}(\mF')|\leq |\mH(\mS')|=|\mH(\mS)\cup\{S_0\}|=|\mH(\mS)|+1=|\text{Sh}(\mF)|+1=|\mF| +1$.
	\item[(iv)] Prove that $\mF \subseteq \mF'$ and $|\mF'| = |\mF| +1$.
\end{itemize}
As we will see after the proof of Theorem~\ref{thm::elimination_main}, one cannot simply take any $S_0 \in \mS$. Another issue is, that it is not clear how to extend the function $h$. We think that for $v\in [n]\setminus S_0$ a natural choice would be to set $h(S_0 \cup \{ v\})$ equal to either $h(S_0)$ or to $h(S_0) \cup \{v\}$.

\subsection{Proof of Theorem~\ref{thm::elimination_main}}

The following proposition establishes that given any Sperner family $\mS\subseteq 2^{[n]}$ and a fixed set $A \subseteq [n]$ the family $\mF(\mS,h_A)$ is s-extremal. 

\begin{proposition}\label{fixedAlemma}
Let $\mathcal{S}\subseteq 2^{[n]}$ be a Sperner family and $A \subseteq [n]$ a fixed set. Furthermore let $h_A:\mathcal{S}\rightarrow 2^{[n]}$ be defined as $h_A(S)=S\cap A$. Then $\mF=\mathcal{F}(\mathcal{S},h_A)$ is s-extremal and $\text{Sh}(\mathcal{F})=\mathcal{H}(\mathcal{S})$.
\end{proposition}
\begin{proof}
As before let $\mathcal{S} = \{S_1,...,S_N \}$ and for $i\in [N]$ put $H_i= h_A(S_i)=S_i\cap A$. For the proof only note that in this case, for every $1 \leq i,j \leq N$ we have \[S_j \cap  H_i = S_j \cap S_i\cap A =S_i \cap S_j\cap A= S_i \cap H_j,\]  
i.e. $\mathbb{I}_{i,j}=1$. In this case $1-\prod_{i\neq j\in I}\mathbb{I}_{i,j}=0$ for every $I\subseteq [N]$, and so 
\[ \sum_{I\subseteq [N]}(-1)^{|I|+1}\left(1-\prod_{i\neq j\in I}\mathbb{I}_{i,j}\right)\left|\bigcap_{i\in I}\mathcal{P}_{S_i}\right|=0.\] 
As we have seen in Subsection~\ref{equiv equation}, equivalently this means that $|\mathcal{H}(\mathcal{S})|=|\mathcal{F}(\mathcal{S},h)|$, and so by Proposition~\ref{prop::s-extremal} $\mF=\mathcal{F}(\mathcal{S},h_A)$ is s-extremal and $\text{Sh}(\mathcal{F})=\mathcal{H}(\mathcal{S})$.
\end{proof}

 A natural first question is that perhaps the converse is also true, i.e. every s-extremal family is of this form. Unfortunately this is not the case, as shown by the following example.

\begin{example}\label{counterexmpl}
	
Let $n=3$ and $\mathcal{S} = \{ S_1, S_2, S_3\}$, where $S_1 = \{1,2 \}$, $S_2 = \{ 1,3 \}$  and $S_3= \{ 2,3 \} $. Furthermore let $h(S_1)=H_1 = \{ 1 \}$, $h(S_2)=H_2 = \emptyset $ and $h(S_3)=H_3 = \emptyset $. Then
\begin{equation*}
\begin{array}{lll}
\mathcal{P}_1=\mathcal{P}_{S_1}= \{ \{1,2\}, \{1,2,3\} \},& \mathcal{P}_2=\mathcal{P}_{S_2} = \{ \{1,3\}, \{1,2,3\} \},& \mathcal{P}_3=\mathcal{P}_{S_3} = \{ \{2,3\},\{1,2,3\} \},\\
\mathcal{Q}_1=\mathcal{Q}_{S_1,H_1}=\{\{1\},\{1,3\}\},& \mathcal{Q}_2=\mathcal{Q}_{S_2,H_2}=\{\emptyset,\{2\}\},& \mathcal{Q}_3=\mathcal{Q}_{S_3,H_3}=\{\emptyset,\{1\}\}
\end{array}
\end{equation*}
and so 
\begin{equation*}
\mathcal{F}(\mathcal{S},h)= 2^{[3]}\setminus \left(\mathcal{Q}_1 \cup \mathcal{Q}_2 \cup \mathcal{Q}_3 \right) = \{\{ 3\}, \{1,2\}, \{2,3\}, \{1,2,3\} \}
\end{equation*}
and 
\begin{equation*}
\mathcal{H}(\mathcal{S})=2^{[3]}\setminus (\mathcal{P}_1 \cup \mathcal{P}_2 \cup \mathcal{P}_3) =\{\emptyset,\{1\},\{2\},\{3\}\}.
\end{equation*}
As both have size $4$, by Proposition~\ref{prop::s-extremal} $\mathcal{F}(\mathcal{S},h)$ is s-extremal and $\text{Sh}(\mathcal{F}(\mathcal{S},h))=\mathcal{H}(\mathcal{S})$. However it is easily seen that there is no $A\subseteq [3]$ such that $h=h_A$ would hold. 
\end{example}
We are now in a position to show that families of the form $\mF(\mS,h_A)$ satisfy Conjecture~\ref{conj}, i.e. to prove Theorem~\ref{thm::elimination_main}.

\begin{proof}[Proof of Theorem~\ref{thm::elimination_main}]
Recall that by Proposition~\ref{fixedAlemma} $\mathcal{F}=\mathcal{F}(\mathcal{S},h_A)$ is s-extremal and $\text{Sh}(\mathcal{F})=\mathcal{H}(\mathcal{S})$. Pick an arbitrary $S_0 \in \mathcal{S}$ and let $H_0=S_0\cap A$. Then there exists a unique (possibly empty) family $\{S_1',...,S_k' \} \subseteq \{ S_0 \cup \{v\}\; : \; v\in [n]\setminus S_0 \}$ such that $\mathcal{S} ' = \mathcal{S} \setminus \{S_0\} \cup \{ S_1',...,S_k'\}$ is again a Sperner familiy and $\mathcal{H(\mathcal{S'})}=\mathcal{H}(\mathcal{S})\cup \{S_0 \}$. For $i\in [k]$ let $H_i'=h_A(S_i')=S_i'\cap A$ and consider the family $\mathcal{F}'=\mathcal{F}(\mathcal{S}',h_A)$. Again, by Proposition~\ref{fixedAlemma}, $\mathcal{F}'$ is s-extremal and $\text{Sh}(\mathcal{F}')=\mathcal{H}(\mathcal{S}')$. In particular we have $|\mF'|=|\text{Sh}(\mathcal{F}')|=|\mathcal{H(\mathcal{S'})}|=|\mathcal{H}(\mathcal{S})\cup \{S_0 \}|=|\mathcal{H}(\mathcal{S})|+1=|\text{Sh}(\mathcal{F})|+1= |\mathcal{F}|+1$. Accordingly all that remains to be shown to prove the theorem is that $\mathcal{F} \subseteq \mathcal{F}'$, since in that case the unique set $F$ in $\mathcal{F}'\setminus \mathcal{F}$ is a good choice. To see this, first note that by the choice of $S_i'$ and $H_i'$ for $i\in [k]$ we have $\mathcal{Q}_{S_i', H_i'} \subseteq \mathcal{Q}_{S_0,H_0}$ and hence 
\begin{equation*}
\bigcup_{i=1}^k \mathcal{Q}_{S_i', H_i'} \subseteq \mathcal{Q}_{S_0,H_0}.
\end{equation*}
However in this case 
\begin{equation*}
\mathcal{F}  = \Bigg(2^{[n]}\setminus \bigcup_{\begin{footnotesize}
\substack{S \in \mathcal{S} \\ S \neq S_0}
\end{footnotesize}} \mathcal{Q}_{S,h_A(S)} \Bigg) \setminus \mathcal{Q}_{S_0,H_0}
 \subseteq \Bigg(2^{[n]}\setminus \bigcup_{\begin{footnotesize}
\substack{S \in \mathcal{S} \\ S \neq S_0}
\end{footnotesize}} \mathcal{Q}_{S,h_A(S)}\Bigg)\setminus \bigcup_{i=1}^k \mathcal{Q}_{S_i',H_i'} = 2^{[n]}\setminus \bigcup_{S \in \mathcal{S}'} \mathcal{Q}_{S,h_A(S)} =\mathcal{F}',
\end{equation*}
as desired.
\end{proof}

Theorem~\ref{thm::elimination_main} solves only a further special case of Conjecture \ref{conj}, so the conjecture remains wide open in general. However the approach presented offers a possible way to tackle it. 

Note that the set $S_0$ in the proof was arbitrary. In general, as already mentioned after the outline of the approach, one cannot take any $S_0 \in \mS$.

\addtocounter{theorem}{-1}

\begin{example}[continued]
Consider the set system introduced earlier in Example~\ref{counterexmpl}. If we take any $S_0\in \mathcal{S}$ then already with $\mathcal{S}'=\mathcal{S}\setminus \{S_0\}$ we have $\mathcal{H}(\mathcal{S}')=\mathcal{H}(\mathcal{S})\cup\{S_0\}$. However if we were to choose $S_0=S_3=\{2,3\}$, then the resulting $\mathcal{F}'$ would be the same as $\mathcal{F}$. In the special case, when we have $h=h_A$ for some $A\subseteq [n]$, this is not possible by the s-extremality of $\mathcal{F}'$, which is guaranteed by Proposition~\ref{fixedAlemma}. Here we remark that $\mathcal{F}=\mathcal{F}'$ does not contradict with the uniqueness of $\mathcal{S}$ and $h$, as for $\mathcal{S}'$ we have that $\text{Sh}(\mathcal{F}')=\text{Sh}(\mathcal{F})\subsetneq \mathcal{H}(\mathcal{S}')$. In the above example for instance $\text{Sh}(\mathcal{F})=\mathcal{H}(\mathcal{S})=\{\emptyset,\{1\},\{2\},\{3\}\}\subsetneq \{\emptyset,\{1\},\{2\},\{3\},\{2,3\}\}=\mathcal{H}(\mathcal{S'})$. Let us also mention that on the other hand $S_1$ and $S_2$ are both good choices for $S_0$.
\end{example}

Accordingly the main issue here is to rule out the possibility $\mathcal{F}=\mathcal{F}'$ by choosing $S_0$ and the new values for $h$ carefully.  Note that to prove the conjecture we need only one good instance.

%%%%%%%%%%%%%%%%%%%%
%%%%%%%%%%%%%%%%%%%%
%%%%%%%%%%%%%%%%%%%%  Small Sperner families
%%%%%%%%%%%%%%%%%%%%
%%%%%%%%%%%%%%%%%%%%

\section{Small Sperner families and an equivalent form of Conjecutre~\ref{conj}}\label{sec::small_sperner}

In this section we will prove that Conjecture~\ref{conj} holds for those s-extremal families whose corresponding Sperner family has size at most 4. To do so, we first state an equivalent version of Conjecture~\ref{conj}.

\begin{conjecture}\label{conj::equiv}
Let $\mS \subseteq 2^{[n]}$ be a Sperner family and $h:\mathcal{S}\rightarrow 2^{[n]}$ a function such that $h(S)\subseteq S$ for every $S\in \mathcal{S}$. Then there exists $S_0 \in \mS$ such that
\begin{equation*}
\mQ_{S_0,h(S_0)} \not\subseteq \bigcup_{S \in \mathcal{S}\setminus\{S_0\}} \mQ_{S,h(S)}.
\end{equation*}
\end{conjecture}

Note that we have already seen that any s-extremal family $\mF \subseteq 2^{[n]}$ is of the form $\mF (\mS, h)$ with $\text{Sh}(\mF) = \mH(\mS)$ for a unique Sperner family $\mS$ and function $h$. We will now show that Conjecture~\ref{conj} and Conjecture~\ref{conj::equiv} for s-extremal families are indeed equivalent.
\begin{lemma}\label{lem:Conj-equiv}
Let $\mF \subsetneq 2^{[n]}$ be s-extremal of the form $\mF (\mS, h)$ with $\text{Sh}(\mF) = \mH(\mS)$ for a Sperner family $\mathcal{S}\subseteq 2^{[n]}$ and a function $h:\mathcal{S}\rightarrow 2^{[n]}$ such that $h(S)\subseteq S$ for every $S\in \mathcal{S}$.. Then there exists $F \in 2^{[n]}\setminus \mF$ such that $\mF' = \mF \cup \{F\}$ is s-extremal if and only if there exists $S_0 \in \mS$ such that $\mQ_{S_0,h(S_0)} \not\subseteq \bigcup_{S \in \mathcal{S}\setminus\{S_0\}} \mQ_{S,h(S)}$.
\end{lemma}
\begin{proof}
Assume first that there exists $F \in 2^{[n]}\setminus \mF$ such that $\mF' = \mF \cup \{ F\}$ is s-extremal. In particular, we have $|\mF|=|\text{Sh}(\mF)|<|\text{Sh}(\mF')|=|\mF'|=|\mF|+1$  and so by the monotonicity of the family of shattered sets there must exist a set $S_0\in \mS$ such that $\text{Sh}(\mF')=\text{Sh}(\mF)\cup\{S_0\}$. From the proof of Lemma~\ref{unique lemma} we know that $h(S_0)$ is the unique subset of $S_0$ that could not be obtained as an intersection of elements of $\mF$ with $S_0$. This implies that $S_0\in \text{Sh}(\mF')$ is only possible if $S_0\cap F=h(S_0)$, i.e $F\in \mQ_{S_0,h(S_0)}$. However, on the other hand we have $F\notin \bigcup_{S \in \mathcal{S}\setminus\{S_0\}} \mQ_{S,h(S)}$ implying $\mQ_{S_0,h(S_0)} \not\subseteq \bigcup_{S \in \mathcal{S}\setminus\{S_0\}} \mQ_{S,h(S)}$. Indeed if there were a set $S \in \mS\setminus \{S_0\}$ with $F \in \mQ_{S,h(S)}$, then $\mF'$ would shatter $S$ too, contradicting $|\text{Sh}(\mF')| = |\text{Sh}(\mF)|+1$.

\medskip

Conversely, suppose there exists $S_0 \in \mS$ such that $\mQ_{S_0,h(S_0)} \not\subseteq \bigcup_{S \in \mathcal{S}\setminus\{S_0\}} \mQ_{S,h(S)}$ and choose a set $F\in \mQ_{S_0,h(S_0)} \setminus \left(\bigcup_{S \in \mathcal{S}\setminus\{S_0\}} \mQ_{S,h(S)}\right)$. We claim that in this case our new approach described in the previous section is applicable with $S_0$ and we can extend the function $h$ so that it gives that $\mF \cup \{F\}$ is s-extremal. Indeed, let, as in the proof of Theorem \ref{thm::elimination_main}, $\{S_1',...,S_k' \} \subseteq \{ S_0 \cup \{v\}\; : \; v\in [n]\setminus S_0 \}$ be the unique family such that $\mathcal{S} ' = \mathcal{S} \setminus \{S_0\} \cup \{ S_1',...,S_k'\}$ is again a Sperner family and $\mathcal{H(\mathcal{S'})}=\mathcal{H}(\mathcal{S})\cup \{S_0 \}$. Now for $i\in [k]$ put $v_i=S_i'\setminus S_0$ and choose $H_i'=h(S_i')$ to be $h(S_0)$ or 
$h(S_0)\cup\{v_i\}$ so that we have $F\notin \mQ_{S_i',H_i'}$. Note that $\mQ_{S_i',h(S_0)}$ and $\mQ_{S_i',h(S_0)\cup\{v_i\}}$ are complementary subcubes of $\mQ_{S_0,h(S_0)}$ so exactly one of the above choices will be appropriate. Now consider the family $\mF(\mS',h)$ for which, just as in the proof of Theorem~\ref{thm::elimination_main}, we have that $\mF\subseteq \mF(\mS',h)$.  Moreover by the Sauer-Shelah lemma, Claim~\ref{fshsh} and the s-extremality of $\mF$ we also get $|\mF(\mS',h)|\leq |\text{Sh}(\mF(\mS',h))|\leq |\mH(\mS')|=|\mH(\mS)\cup\{S_0\}|=|\mH(\mS)|+1=|\text{Sh}(\mF)|+1=|\mF|+1$. However by the choice of $F$ and the $H_i'$s we have $F\notin \bigcup_{S\in \mS\setminus\{S_0\}}\mQ_{S,h(S)}$ and also $F\notin\bigcup_{i=1}^k\mQ_{S_i',h(S_i')}$. Hence $F\in \mF(\mS',h)\setminus \mF$ which is possible only if $\mF(\mS',h)=\mF\cup \{F\}$ and if there is equality everywhere in the earlier series of inequalities, in particular  $\mF(\mS',h)=\mF\cup \{F\}$ is s-extremal.
\end{proof}
\noindent The following Claim is useful when one works with the above version of the conjecture.
\begin{claim}\label{claim:Q-sperner}
Let $\mS \subseteq 2^{[n]}$ be a Sperner family and $h:\mathcal{S}\rightarrow 2^{[n]}$ a function such that $h(S)\subseteq S$ for every $S\in \mathcal{S}$. Then for every $S_1 \neq S_2\in\mS$ we have $\mQ_{S_1,h(S_1)} \not\subseteq \mQ_{S_2,h(S_2)}$.
\end{claim}
\begin{proof}
Suppose the claim is false, i.e. there are sets $S_1 \neq S_2\in\mS$ such that $\mQ_{S_1,h(S_1)} \subseteq \mQ_{S_2,h(S_2)}$. This in particular means that $h(S_1),h(S_1)\cup([n]\setminus S_1)\in \mQ_{S_1,h(S_1)}$ also belong to $\mQ_{S_2,h(S_2)}$. However this means that $h(S_1)\cap S_2= h(S_2)$ and $\big(h(S_1)\cup([n]\setminus S_1)\big)\cap S_2=h(S_2)$ which, as $h(S_1)$ and $[n]\setminus S_1$ are disjoint, is possible only if $S_2 \subseteq S_1$, contradicting the fact that  $\mS$ is a Sperner family.
\end{proof}
Given a Sperner family $\mS=\{S_1,\dots,S_N\}\subseteq 2^{[n]}$ and a function $h:\mathcal{S}\rightarrow 2^{[n]}$ such that $h(S)\subseteq S$ for every $S\in \mathcal{S}$ let us define an auxiliary graph $G_{\mS,h}$. For ease of notation for $i\in [N]$ put $h(S_i)=H_i$ and $\mQ_{S_i,H_i}=Q_i$.  Then the vertex set of $G_{\mS,h}$ is given by $V(G_{\mS,h}) = \{ (S_i,H_i)\; : i \in [N] \}$ and we join two vertices $(S_i,H_i)$ and $(S_j,H_j)$ if $S_i \cap H_j = S_j \cap H_i$. Note that by earlier remarks if there is an edge between $(S_i,H_i)$ and $(S_j,H_j)$, then the corresponding cubes intersect and are disjoint otherwise.

We will now state and prove some basic facts about $G_{\mS,h}$ related to Conjecture~\ref{conj::equiv}.

\begin{claim}\label{claim:degree1}
Let $\mS=\{S_1,\dots,S_N\} \subseteq 2^{[n]}$ be a Sperner family and $h:\mathcal{S}\rightarrow 2^{[n]}$ a function such that $H_i=h(S_i)\subseteq S_i$ for every $i\in[N]$. Then if there is a vertex of degree at most $1$ in $G_{\mS,h}$, then $(\mS,h)$ satisfies Conjecture~\ref{conj::equiv}.
\end{claim}
\begin{proof}
For ease of notation for $i\in [N]$ again put $Q_{S_i,H_i}=Q_i$.

In case there exists an isolated vertex, $(S_i,H_i)$ say, it is clear that $ \mQ_i \not\subseteq \bigcup_{j \in [N]\setminus \{i\}} \mQ_j$, since $\mQ_i$ is disjoint from all other $\mQ_j$.

Otherwise there exists a vertex of degree 1. After possible relabelling, we may assume $(S_1,H_1)$ has degree $1$, and that $(S_2,H_2)$ is its unique neighbour. However now $\mQ_1$ is disjoint from all $\mQ_j$, $j \geq 3$, so  $\mQ_1 \subseteq \bigcup_{j \geq 2} \mQ_j$ would be only possible if we would have $\mQ_1 \subseteq \mQ_2$ which is impossible by Claim~\ref{claim:Q-sperner}. Hence $(\mS,h)$ really satisfies Conjecture~\ref{conj::equiv}
\end{proof}
\begin{claim}\label{claim:completegraph}
Let $\mS=\{S_1,\dots,S_N\} \subseteq 2^{[n]}$ be a Sperner family and $h:\mathcal{S}\rightarrow 2^{[n]}$ a function such that $H_i=h(S_i)\subseteq S_i$ for every $i\in[N]$. Then if $G_{\mS,h}$ is the complete graph on $N$ vertices then $(\mS,h)$ satisfies Conjecture~\ref{conj::equiv}.
\end{claim}
\begin{proof}
	Since $G_{\mS,h}$ is the complete graph on $N$ vertices, we have $S_i \cap H_j = S_i \cap H_j$ for all $1 \leq i < j \leq N$. We will show that in this case $h = h_A$, where $A = H_1 \cup ... \cup H_N \subseteq [n]$. Then by Theorem~\ref{prop::s-extremal} the family $\mF(\mS,h)$ satisfies Conjecture~\ref{conj} and so by Lemma~\ref{lem:Conj-equiv} we have that $(\mS,h)$ satisfies Conjecture~\ref{conj::equiv}. Recall that $h_A(S_i)=S_i \cap A$. To finish the proof, just note that $h_A(S_i)=S_i \cap A = S_i \cap (H_1\cup ...\cup H_N)= (S_i\cap H_1 ) \cup ...\cup (S_i \cup H_N) = H_i$, since $S_i\cap H_i = H_i$ and for every $j\neq i$  we have $S_i \cap H_j = S_j \cap H_i \subseteq H_i$.
\end{proof}
Now we are in a position to prove Theorem~\ref{thm::conjecture_for_small_sperner}.

\begin{proof}[Proof of Theorem~\ref{thm::conjecture_for_small_sperner}]
Let $\mF=\mF(\mS,h)$ as in the statement. By Lemma~\ref{lem:Conj-equiv} it is enough to prove that $(\mS,h)$ satisfies Conjecture~\ref{conj::equiv}. 

\medskip

If $|\mS| \leq 3$, then $G_{\mS,h}$ either has a vertex of degree at most one or is a clique and so by Claims~\ref{claim:degree1} and \ref{claim:completegraph} $(\mS,h)$ satisfies Conjecture~\ref{conj::equiv}. 	

\medskip

If $|\mS| = 4$, then the only cases not handled by Claims~\ref{claim:degree1} and \ref{claim:completegraph} are when $G_{\mS,h}$ is a $C_4$, i.e. a cycle of length four, or a $K_4^-$, i.e. a complete graph on four vertices minus an edge. As before, for $i\in[n]$ put $Q_i=Q_{S_i,H_i}$.

\medskip

Suppose first that $G_{\mS,h} = C_4$ with vertices labeled so that we have $\mQ_1 \cap \mQ_3 = \emptyset = \mQ_2 \cap \mQ_4$ and assume that $(\mS,h)$ does not satisfy Conjecture~\ref{conj::equiv}.  Then we have  $\mQ_1 , \mQ_3 \subseteq \mQ_2 \cup \mQ_4$ and $\mQ_2, \mQ_4 \subseteq \mQ_1 \cup \mQ_3$ and hence $\mQ_1 \cup \mQ_3 = \mQ_2 \cup \mQ_4$, where on both sides we have disjoint unions. For $1\leq i<j\leq 4$ put $\mQ_{i,j}=\mQ_i\cap\mQ_j$. Then $\mQ_1=\mQ_{1,2}\cup\mQ_{1,4}$, $\mQ_2=\mQ_{1,2}\cup\mQ_{2,3}$, $\mQ_3=\mQ_{2,3}\cup\mQ_{3,4}$ and $\mQ_4=\mQ_{1,4}\cup\mQ_{3,4}$ where in each case we have a disjoint union of nonempty cubes. Now this is possible only if all the $\mQ_i$'s are $d-1$ dimensional subcubes of some $d$ dimensional cube, in which case we necessarily have $S_1=S_3$ and $S_2=S_4$ - a contradiction.

\medskip

Now assume that $G_{\mS,h} = K_4^-$ with vertices labeled so that we have $\mQ_1 \cap \mQ_3 = \emptyset$ and assume that $(\mS,h)$ does not satisfy Conjecture~\ref{conj::equiv}. Then in particular $\mQ_1 \subseteq \mQ_2 \cup \mQ_4$. Let as before $\mQ_{1,2} = \mQ_1 \cap \mQ_2$ and $\mQ_{1,4} = \mQ_1 \cap \mQ_4$. Then both $\mQ_{1,2}$ and $\mQ_{1,4}$ are subcubes of $\mQ_1$ and $\mQ_2' \cup \mQ_4' = \mQ_1$. There are two cases to consider. 

Either we have $\mQ_{1,2} = \mQ_1$ or $\mQ_{1,4} = \mQ_1$, in which case we have $\mQ_1 \subseteq \mQ_2$ or $\mQ_1 \subseteq \mQ_4$, contradicting Claim~\ref{claim:Q-sperner}.

Otherwise $\mQ_{1,2}$ and $\mQ_{1,4}$ are two disjoint subcubes of $\mQ_1$ of one smaller dimension. In particular, there exists a direction $i\in [n]$ that distinguishes them. But the same direction is then distinguishing $\mQ_2$ and $\mQ_4$ implying $\mQ_2 \cap \mQ_4 = \emptyset$, which is a contradiction as the only disjoint pair of cubes were $\mQ_1$ and $\mQ_3$.
\end{proof}

%%%%%%%%%%%%%%%%%%%%%%%%%%

%%%%%%%%%%%%%%%%%%%%
%%%%%%%%%%%%%%%%%%%%
%%%%%%%%%%%%%%%%%%%%  Grobner bases
%%%%%%%%%%%%%%%%%%%%
%%%%%%%%%%%%%%%%%%%%

\section{Gr\"obner bases and s-extremal families}\label{sec::preliminaries_groebner}

Let $\mathbb{F}$ be an arbitrary field and let $\mathbb{F}[x_1,...,x_n] = \mathbb{F}[\textbf{x}]$ be the polynomial ring over $\mathbb{F}$ with variables $x_1,...,x_n$. Given some set $F\subseteq [n]$, let $v_F \in \{0,1 \}^n$ be its \emph{characteristic vector}, i.e. the $i$-th coordinate of $v_F$ is $1$ if $i\in F$ and $0$ otherwise. Then we can identify a set system $\mathcal{F} \subseteq 2^{[n]}$ with the vector system 
\begin{equation*}
\mathcal{V}(\mathcal{F}) = \{ v_F\ :\ F\in \mathcal{F} \} \subseteq \{0,1 \}^n \subseteq \mathbb{F}^n,
\end{equation*}
and associate to $\mathcal{F}$ a polynomial ideal $I(\mathcal{F}) \unlhd \mathbb{F}[\textbf{x}]$, where 
\begin{equation*}
I(\mathcal{F})=I(\mathcal{V}(\mathcal{F}))= \{ f\in \mathbb{F}[\textbf{x}]\ :\ f(v_F)=0 \; \forall\; F \in \mathcal{F} \},
\end{equation*}
i.e. $I(\mathcal{F})$ is the vanishing ideal of the set of characteristic vectors of the elements of $\mathcal{F}$. Note that we always have $\{x_i^2 -x_i\ :\ i\in [n] \} \subseteq I(\mathcal{F})$. For more details about vanishing ideals of finite point sets see e.g. \cite{RM11}.

If one works with polynomial ideals, it is useful to have a nice ideal basis. Such nice bases are given by the so-called \emph{Gr\"obner bases}, which we will now briefly define. For more details the interested reader may consult e.g. \cite{AL94}. A total order $\prec$ on the monomials in $\mathbb{F}[\textbf{x}]$ is a \emph{term order} if $1$ is the minimal element of $\prec$ and $\prec$ is compatible with multiplication with monomials. One well-known and important term order is the \emph{lexicographic (lex) order}. Here one has $x_1^{w_1}...x_n^{w_n} \prec_{\text{lex}} x_1^{u_1}...x_n^{u_n}$ if and only if for the smallest index $k$ with $w_k \neq u_k$ one has $w_k < u_k$. One can build a lex order based on other orderings of the variables as well, so altogether we have $n!$ different lex orders. Given some term order $\prec$ and a nonzero polynomial $f \in \mathbb{F}[\textbf{x}]$, the \emph{leading monomial} $\text{lm}(f)$ of $f$, is the largest monomial (with respect to $\prec$) appearing with non-zero coefficient in the canonical form of $f$.

Now let $I \unlhd \mathbb{F}[\textbf{x}]$ be an ideal and $\prec$ a term order. A finite subset $\mathbb{G} \subseteq I$ is called a \emph{Gr\"obner basis of I} with respect to $\prec$ if for every nonzero polynomial $f\in I$ there exists a $g\in \mathbb{G}$ such that $\text{lm}(g)$ divides $\text{lm}(f)$. Gr\"obner bases have a lot of nice properties, among others we know that every non-zero ideal $I\unlhd \mathbb{F}[\textbf{x}]$ has a Gr\"obner basis for every term order, and if $\mathbb{G}$ is a Gr\"obner basis of $I$ for some term order, then $\mathbb{G}$ generates $I$ as an ideal as well, i.e. $I=\langle\mathbb{G}\rangle$.

For a subset $H \subseteq [n]$, set $\textbf{x}_H = \prod_{i \in H}x_i$, and given a pair of sets $H \subseteq S \subseteq [n]$ we then define the polynomial
\begin{equation*}
f_{S,H}(\textbf{x}) = \textbf{x}_H \cdot \prod_{i \in S\setminus H}(x_i -1).
\end{equation*}

A nice property of these polynomials is that for a set $F\subseteq [n]$ we have $f_{S,H}(v_F) \neq 0$ if and only if $F\cap S = H$.

For a Sperner family $\mathcal{S}$ and a function $h:\mathcal{S}\rightarrow 2^{[n]}$ such that $h(S)\subseteq S$ for every $S\in \mathcal{S}$ put
\begin{equation*}
\mathbb{G}(\mathcal{S},h) = \{f_{S,h(S)}\ :\ S\in \mathcal{S} \} \cup \{x_i^2-x_i\ :\ i\in [n] \}.
\end{equation*}

Now we are in a position to state the connection between s-extremal families and the theory of Gr\"obner bases.  
\begin{theorem}[\cite{RM11}]\label{ExtremalGroebner}
$\mathcal{F} \subseteq 2^{[n]}$ is s-extremal if and only if there are polynomials of the form $f_{S,H}$, which together with $\{x_i^2-x_i\; : \; i \in [n] \}$ form a Gr\"obner basis of $I(\mathcal{F})$ for all term orders. 
\end{theorem}
We remark that from the proof of Theorem~\ref{ExtremalGroebner} from \cite{RM11} it also follows that if there is a suitable Gr\"obner basis for one particular term order then $\mathcal{F}$ is already s-extremal. Also if $\mF$ is s-extremal then $\mathbb{G}(\mathcal{S},h)$ is such a Gröbner basis for the Sperner family $\mS$ and function $h$ guaranteed by Lemma~\ref{unique lemma}.

Using  the approach given by Proposition~\ref{prop::s-extremal}, we can now state and prove our main result of this section. The proof requires a basic knowledge of commutative algebra, in particular ideal theory, for more details we refer the reader to the classical book of Atiyah and MacDondald~\cite{atiyahmacdonald1994}.

\begin{theorem}\label{thm::Grobner_equation}                 
Let $\mathcal{S}$ be a Sperner family and $h:\mathcal{S}\rightarrow 2^{[n]}$ a function  such that $h(S)\subseteq S$ for every $S\in \mathcal{S}$. Then $\mathbb{G}=\mathbb{G}(\mathcal{S},h)$ is a Gr\"obner basis (of $\langle\mathbb{G}\rangle$) for some term order if and only if 
\begin{equation*}
|\mathcal{H}(\mathcal{S})|=|\mathcal{F}(\mathcal{S},h)|.
\end{equation*}
\end{theorem}

\begin{proof}
Suppose first that $\mathbb{G}$ is a Gr\"obner basis for some term order $\prec$. We start by showing that the ideal generated by $\mathbb{G}$ is a radical ideal, i.e. $\langle \mathbb{G} \rangle = \sqrt{\langle \mathbb{G} \rangle}$. To see this, first note that clearly $J = \langle x_i^2 - x_i\; : i \in [n] \rangle \subseteq \langle \mathbb{G} \rangle$. Now a basic fact from Algebra states that $\langle \mathbb{G} \rangle$ is a radical ideal in $\mathbb{F}[\bx]$ if and only if $\langle \mathbb{G} \rangle / J$ is a radical ideal in $\mathbb{F}[\bx] / J$. However, $\mathbb{F}[\bx]/J $ is isomorphic to $\mathbb{F}^{2^n}$, because both are isomorphic to the ring of all functions from $\{0,1 \}^n$ to $\mathbb{F}$. Using the fact that the only ideals of a field are the zero ideal and the field itself, and that the only ideals in a finite cartesian product of rings are products of ideals, one easily verifies that every ideal in $\mathbb{F}^{2^n}$ is the intersection of maximal ideals. This in turn implies that in $\mathbb{F}[\bx]/J$ every ideal is a radical ideal and so in particular $\langle \mathbb{G} \rangle /J $ is. Hence $\langle \mathbb{G} \rangle$ is a radical ideal and thus, since $J \subseteq \langle \mathbb{G} \rangle$, is a vanishing ideal of some finite set in $\{ 0,1\}^n$, which can be clearly understood as the set of characteristic vectors of some family $\mF\subseteq 2^{[n]}$. However, by the earlier mentioned properties of the $f_{S,h(S)}$ polynomials, $\mF$ is then  precisely $\mF(\mS,h)$:

\begin{equation*}
\mathcal{F} =\bigcap_{S\in \mathcal{S}}\{F\ :\ v_F\text{ is a root of }f_{S,h(S)}\}=\bigcap_{S\in \mathcal{S}}\{v_F\ :\ F\cap S\neq h(S)\}
\end{equation*}
\begin{equation*}
=\{ v_F\ :\ F\cap S \neq h(S) \ \forall S\in \mathcal{S} \} = 2^{[n]}\setminus \bigcup_{S\in \mathcal{S}} \mathcal{Q}_{S,h(S)} = \mathcal{F}(\mathcal{S},h).
\end{equation*}
Thus $\langle\mathbb{G}\rangle = I(\mathcal{F}(\mathcal{S},h))$ and so, by Theorem~\ref{ExtremalGroebner}, $\mathcal{F}(\mathcal{S},h)$ is s-extremal. However, according to the proof of Lemma~\ref{unique lemma}, in this case $\mS$ is the collection of all minimal sets not shattered by $\mF(\mS,h)$ and so we have $\text{Sh}(\mathcal{F}(\mathcal{S},h))=\mathcal{H}(\mathcal{S})$, implying $|\mathcal{F}(\mathcal{S},h)|=|\text{Sh}(\mF(\mS,h))|=|\mathcal{H}(\mathcal{S})|$.

\medskip

Now suppose $|\mathcal{H}(\mathcal{S})|=|\mathcal{F}(\mathcal{S},h)|$. In terms of Theorem~\ref{ExtremalGroebner} it is enough to show that $\mathcal{F}(\mathcal{S},h)$ is s-extremal, however this just follows from Proposition~\ref{prop::s-extremal}.
\end{proof}

%%%%%%%%%%%%%%%%%%%%
%%%%%%%%%%%%%%%%%%%%
%%%%%%%%%%%%%%%%%%%%  Concluding Remarks
%%%%%%%%%%%%%%%%%%%%
%%%%%%%%%%%%%%%%%%%%

\section{Concluding remarks}

In this paper we proved the elimination conjecture for some special cases and presented a new approach that we hope should work to prove the conjecture in full generality.

As already noted, given an s-extremal family $\mF\subseteq 2^{[n]}$ with its unique Sperner family $\mS$ and function $h$ one of the main problems in our new approach is to identify a suitable set $S_0\in \mS$ to start with and to extend the function $h$ to the new sets in the Sperner family. In this direction it would be nice to solve the following problem.
\begin{problem}
For a given Sperner family $\mS \subseteq 2^{[n]}$ determine all possible functions $h:\mS\rightarrow 2^{[n]}$ with $h(S)\subseteq S$ for every $S\in \mS$ such that the resulting set system $\mF(\mS,h)$ is s-extremal with $\text{Sh}(\mF(\mS,h))=\mH(\mS)$.
\end{problem}
To end with, we would like to remark in connection with Conjecture~\ref{conj::equiv} that it covers a much more general case then necessarily needed for our purposes, as in relation with Conjecture~\ref{conj} we are only interested in $(\mS,h)$ pairs that come from an s-extremal family $\mF$. So even if Conjecture~\ref{conj::equiv} fails to be true in its complete generality, it might remain valid in the special case of $(\mS,h)$ pairs coming from s-extremal families.

%%%%%%%%%%%%%%%%%%
%%%%%%%%%%%%%%%%%% BIBLIOGRAPHY
%%%%%%%%%%%%%%%%%%
%%%%%%%%%%%%%%%%%%

\bibliography{bib}
\bibliographystyle{plain}

\end{document}